\documentclass[12pt, reqno]{amsart}
\usepackage{amsmath, amsthm, amscd, amsfonts, amssymb, graphicx, color}

 \usepackage{setspace}
	\usepackage{graphicx}
	 
	\usepackage{mathrsfs}
	\usepackage{yfonts}
 
	\usepackage{graphicx} 
 
\usepackage{url}
\usepackage{enumerate}
\usepackage[bookmarksnumbered, colorlinks, plainpages]{hyperref}
\hypersetup{colorlinks=true,linkcolor=red, anchorcolor=green,
citecolor=cyan, urlcolor=red, filecolor=magenta, pdftoolbar=true}

\newtheorem{theorem}{Theorem}[section]
\newtheorem{lemma}[theorem]{Lemma}

\newtheorem{corollary}[theorem]{Corollary}
\theoremstyle{definition}

\theoremstyle{remark}
\newtheorem{remark}[theorem]{Remark}
\numberwithin{equation}{section}

  \DeclareMathOperator{\spe}{sp}

  \def\etal{et al.\,}
\begin{document}
\setcounter{page}{1}

 \title[On the generalized mixed Schwarz inequality]
{On the generalized mixed Schwarz inequality}
\author[M.W. Alomari]{Mohammad .W. Alomari}
 
\address{Department of Mathematics, Faculty of Science and Information
	Technology, Irbid National University, 2600 Irbid 21110, Jordan.}
\email{\textcolor[rgb]{0.00,0.00,0.84}{mwomath@gmail.com}}

\date{\today}


 
\date{Received: xxxxxx; Revised: yyyyyy; Accepted: zzzzzz.}

\begin{abstract}
In this work,  an  extension of  the generalized mixed Schwarz
inequality   is proved. A companion of the generalized mixed
Schwarz inequality is established by merging   both Cartesian and
Polar decompositions of operators. Based on that some numerical
radius inequalities are proved.

\end{abstract}

\maketitle

\section{Introduction}

Let $\mathscr{B}\left( \mathscr{H}\right) $ be the Banach algebra
of all bounded linear operators defined on a complex Hilbert space
$\left( \mathscr{H};\left\langle \cdot ,\cdot \right\rangle
\right)$  with the identity operator  $1_\mathscr{H}$ in
$\mathscr{B}\left( \mathscr{H}\right) $.

The Schwarz inequality for positive operators reads that if $A$ is a positive operator in $\mathscr{B}\left(\mathscr{H}\right)$, then
\begin{align}
\left| {\left\langle {Ax,y} \right\rangle} \right|  ^2  \le \left\langle {A x,x} \right\rangle \left\langle { A y,y} \right\rangle, \qquad 0\le \alpha \le 1. \label{eq1.1}
\end{align}
for any   vectors $x,y\in \mathscr{H}$.

In 1951, Reid \cite{R} proved an inequality which in some senses
considered a variant of Schwarz inequality. In fact, he proved
that for all operators $A\in \mathscr{B}\left( \mathscr{H}\right)
$ such that $A$ is positive and $AB$ is selfadjoint then
\begin{align}
\left| {\left\langle {ABx,y} \right\rangle} \right|  \le \|B\|
\left\langle {A x,x} \right\rangle, \label{eq1.2}
\end{align}
for all $x\in \mathscr{H}$. In \cite{H}, Halmos presented his
stronger version of Reid inequality \eqref{eq1.2} by replacing
$r\left(B\right)$ instead of $\|B\|$.

In 1952, Kato  \cite{TK} introduced a companion inequality of
\eqref{eq1.1}, called  the mixed Schwarz inequality,  which
asserts
\begin{align}
\left| {\left\langle {Ax,y} \right\rangle} \right|  ^2  \le \left\langle {\left| A \right|^{2\alpha } x,x} \right\rangle \left\langle {\left| {A^* } \right|^{2\left( {1 - \alpha } \right)} y,y} \right\rangle, \qquad 0\le \alpha \le 1. \label{eq1.3}
\end{align}
for all positive operators $A\in \mathscr{B}\left( \mathscr{H}\right) $ and any vectors $x,y\in \mathscr{H}$, where  $\left|A\right|=\left(A^*A\right)^{1/2}$.

In 1988,  Kittaneh  \cite{FK4} proved  a very interesting extension combining both the Halmos--Reid inequality \eqref{eq1.2} and the   mixed Schwarz inequality \eqref{eq1.3}. His result reads that
\begin{align}
\left| {\left\langle {ABx,y} \right\rangle } \right| \le r\left(B\right)\left\| {f\left( {\left| A \right|} \right)x} \right\|\left\| {g\left( {\left| {A^* } \right|} \right)y} \right\|\label{kittaneh.ineq}
\end{align}
for any   vectors $x,y\in  \mathscr{H} $, where $A,B\in \mathscr{B}\left( \mathscr{H}\right)$ such that $|A|B=B^*|A|$ and    $f,g$ are  nonnegative continuous functions  defined on $\left[0,\infty\right)$ satisfying that $f(t)g(t) =t$ $(t\ge0)$.       Clearly, choose $f(t)=t^{\alpha}$ and $g(t)=t^{1-\alpha}$ with   $B=1_{\mathscr{H}}$ we refer to \eqref{eq1.3}. Moreover, choosing $\alpha=\frac{1}{2}$ some manipulations refer to Halmos version of Reid inequality.

In 2006, Lin and Dragomir \cite{LD} proved the following sequence of inequalities  of Halmos--Ried's type:
\begin{align}
\left. \begin{array}{l}
\left| {\left\langle {Tx,y} \right\rangle } \right|^2  \le r\left( T \right)\left\langle {Tx,x} \right\rangle \left\| y \right\|^2  \\
\\
\left| {\left\langle {TSx,Cy} \right\rangle } \right|  \le r\left( S \right)r\left( C \right)\left\langle {Tx,x} \right\rangle ^{1/2} \left\langle {Ty,y} \right\rangle ^{1/2}  \\
\\
\left| {\left\langle {TSx,y} \right\rangle } \right|  \le r\left( S \right)r\left( C \right)\left\langle {Tx,x} \right\rangle  \\
\\
\left| {\left\langle {Ax,By} \right\rangle } \right|^2  \le r\left( A \right)r\left( B \right)\left\| {Ax} \right\|\left\| {By} \right\|\left\| x \right\|\left\| y \right\| \\
\end{array} \right\}\label{eq1.5}
\end{align}
where $A,B,C,T,S\in \mathscr{B}\left( \mathscr{H}\right) $ such
that $T$ is non-negative operator, $S$ and $C$ are arbitrary
operators, and
$TS$, $TC$, $A$ and $B$ be   selfadjoint operators, for all  vectors $x,y\in \mathscr{H}$. \\

For a bounded linear operator $T$ on a Hilbert space
$\mathscr{H}$, the numerical range $W\left(T\right)$ is the image
of the unit sphere of $\mathscr{H}$ under the quadratic form $x\to
\left\langle {Tx,x} \right\rangle$ associated with the operator.
More precisely,
\begin{align*}
W\left( T \right) = \left\{ {\left\langle {Tx,x} \right\rangle :x
	\in \mathscr{H},\left\| x \right\| = 1} \right\}
\end{align*}
Also, the numerical radius is defined to be
\begin{align*}
w\left( T \right) = \sup \left\{ {\left| \lambda\right|:\lambda
	\in W\left( T \right) } \right\} = \mathop {\sup }\limits_{\left\|
	x \right\| = 1} \left| {\left\langle {Tx,x} \right\rangle }
\right|.
\end{align*}

The spectral radius of an operator $T$ is defined to be
\begin{align*}
r\left( T \right) = \sup \left\{ {\left| \lambda\right|:\lambda
	\in \spe\left( T \right) } \right\}
\end{align*}

We recall that,  the usual operator norm of an operator $T$ is
defined to be
\begin{align*}
\left\| T \right\| = \sup \left\{ {\left\| {Tx} \right\|:x \in
	H,\left\| x \right\| = 1} \right\}.
\end{align*}

It is well known that $w\left(\cdot\right)$ defines an operator
norm on $\mathscr{B}\left( \mathscr{H}\right) $ which is
equivalent to operator norm $\|\cdot\|$. Moreover, we have
\begin{align}
\frac{1}{2}\|T\|\le w\left(T\right) \le \|T\|\label{eq1.6}
\end{align}
for any $T\in \mathscr{B}\left( \mathscr{H}\right)$. The
inequality is sharp.

In 2003, Kittaneh \cite{FK1}  refined the right-hand side of
\eqref{eq1.1}, where he proved that
\begin{align}
w\left(T\right) \le
\frac{1}{2}\left(\|T\|+\|T^2\|^{1/2}\right)\label{eq1.7}
\end{align}
for any  $T\in \mathscr{B}\left( \mathscr{H}\right)$.

After that in 2005, the same author in \cite{FK} proved that
\begin{align}
\frac{1}{4}\|A^*A+AA^*\|\le  w^2\left(A\right) \le
\frac{1}{2}\|A^*A+AA^*\|.\label{eq1.8}
\end{align}
The inequality is sharp. This inequality was also reformulated and generalized in \cite{EF} but in terms of Cartesian
decomposition.

In 2007, Yamazaki \cite{Y} improved \eqref{eq1.7} by proving that
\begin{align}
w\left( T \right) \le \frac{1}{2}\left( {\left\| T \right\| +
	w\left( {\widetilde{T}} \right)} \right) \le \frac{1}{2}\left(
{\left\| T \right\| + \left\| {T^2 } \right\|^{1/2} }
\right)\label{eq1.9}
\end{align}
where $\widetilde{T}=|T|^{1/2}U|T|^{1/2}$ with unitary $U$.

In 2008, Dragomir \cite{D4} used Buzano inequality to improve
\eqref{eq1.1}, where he proved that
\begin{align}
w^2\left( T \right) \le \frac{1}{2}\left( {\left\| T \right\| +
	w\left( {T^2} \right)} \right) \label{eq1.10}
\end{align}
This result was also recently generalized by Sattari \etal in
\cite{SMY}. For more recent results about numerical radius see
\cite{D1}, \cite{F}, \cite{FK2}, \cite{MSS} and the recent
monograph study \cite{D1}. For basic properties of numerical
radius and other related topics the reader may refer to the
classical book of Horn \cite{H2}.\\

In this work,  an  extension of  Kittaneh inequality
\eqref{kittaneh.ineq} and a generalization of Lin-Dragomir version of
Halmos--Ried type inequalities are proved. Namely, we generalize
the Kittaneh inequality \eqref{kittaneh.ineq} which already extend
the mixed Schwarz inequality \eqref{eq1.3} to be in more general
case. A generalization of the obtained result for several
operators is also pointed out. A companion of the generalized
mixed Schwarz inequality (or Kittaneh inequality) in which the
Cartesian decomposition of operators is replaced by the polar
decomposition is also given. As application, some numerical radius
and norm inequalities are established.

\section{The Result(s)}\label{sec2}
This section is divided into two parts; the first part is devoted
to generalize the Kittaneh inequality \eqref{kittaneh.ineq} with other
related consequences. In the second part,
by merging  the  Cartesian and  Polar decompositions of operators, we present a new type of mixed Schwarz inequality called ``the  Mixed hybrid Schwarz inequality ".

\subsection{The Mixed Schwarz inequality\label{sec2.1}}
Let us start with the following elementary result which is a simple consequence of \eqref{eq1.3}.

\begin{lemma}
	\label{lem}Let  $A\in \mathscr{B}\left(
	\mathscr{H}\right)$, then 
	\begin{align}
	\left| {\left\langle {ADu,Cv} \right\rangle } \right|^2  \le\left\langle {D^* f^2 \left( {\left| A \right|} \right)Du,u}
	\right\rangle \left\langle {C^* g^2 \left( {\left| {A^* } \right|}
		\right)Cv,v} \right\rangle \label{eq.prp}
	\end{align}
	for all  vectors $u,v\in  \mathscr{H} $.
\end{lemma}

\begin{proof}
	Since $x,y\in \mathscr{H}$ are arbitrary vectors then there exists
	$u,v\in \mathscr{H}$ respectively; such that $x=Du$ and  $y=Cv$, and this is true for any
	$x,y\in \mathscr{H}$. Therefore by
	setting $B=1_{\mathscr{H}}$ in \eqref{kittaneh.ineq}. Then we have
	\begin{align*}
	\left| {\left\langle {ADu,Cv} \right\rangle } \right|^2 &\le \left\| {f\left( {\left| A \right|} \right)Dx} \right\|^2\left\| {g\left( {\left| {A^* } \right|} \right)Cv} \right\|^2\\
	&\le \left\langle {f\left( {\left| A \right|} \right)Du,f\left(
		{\left| A \right|} \right)Du} \right\rangle \left\langle {g\left(
		{\left| {A^* } \right|} \right)Cv,g\left( {\left| {A^* } \right|}
		\right)Cv} \right\rangle
	\\
	&\le\left\langle {D^* f^2 \left( {\left| A \right|} \right)Du,u}
	\right\rangle \left\langle {C^* g^2 \left( {\left| {A^* } \right|}
		\right)Cv,v} \right\rangle.
	\end{align*}
	for all  vectors $u,v\in  \mathscr{H} $,	which proves the result.
\end{proof}

\begin{remark}
	In particular case, choosing $f(t)=t^\alpha$, $g(t)=t^{1-\alpha}$, $t\ge0$ in \eqref{eq.prp} we get
	\begin{align*}
	\left| {\left\langle {ADu,Cv} \right\rangle } \right|^2  \le\left\langle {D^*  \left| A \right|^{2\alpha} Du,u}
	\right\rangle \left\langle {C^*  \left| {A^* } \right|^{2\left(1-\alpha\right)} Cv,v} \right\rangle. 
	\end{align*}
	In special case, for $\alpha=\frac{1}{2}$ we have
	\begin{align*}
	\left| {\left\langle {ABu,Cv} \right\rangle } \right|^2  \le\left\langle {D^*  \left| A \right|Du,u}
	\right\rangle \left\langle {C^*  \left| {A^* } \right|  Cv,v} \right\rangle 
	\end{align*}
	for all  vectors $x,u\in  \mathscr{H} $.
\end{remark}

Now, we are ready to present our generalization of \eqref{kittaneh.ineq} and the above inequality \eqref{eq.prp} for any bounded linear operators.
\begin{theorem}
	\label{thm1}    Let  $A,B,C\in \mathscr{B}\left(
	\mathscr{H}\right)$ such that $|A|B=B^*|A|$  and
	$|A^*|C=C^*|A^*|$. If $f$ and $g$ are nonnegative continuous
	functions on $\left[0,\infty\right)$ satisfying $f(t)g(t) =t$
	$(t\ge0)$, then
	\begin{align}
	\left| {\left\langle {ABx,Cu} \right\rangle } \right| \le
	r\left(B\right)r\left(C\right)\left\| {f\left( {\left| A \right|}
		\right)x} \right\|\left\| {g\left( {\left| {A^* } \right|}
		\right)u} \right\|\label{eq2.2}
	\end{align}
	for all  vectors $x,u\in  \mathscr{H} $.
\end{theorem}

\begin{proof}
	Our proof is motivated by \cite{FK4}. It's enough to show that the
	inequality
	\begin{align}
	\left| {\left\langle {ABx,Cu} \right\rangle } \right|^{2^n}
	\label{eq2.3} &\le\left\langle {f^2 \left( {\left| A \right|}
		\right)B^{2^n } x,x} \right\rangle \left\langle {f^2 \left(
		{\left| A \right|} \right)x,x} \right\rangle ^{2^{n - 1}  - 1}
	\\
	&\qquad\times \left\langle {g^2 \left( {\left| A^* \right|}
		\right)C^{2^n } u,u} \right\rangle\left\langle {g^2 \left( {\left|
			{A^* } \right|} \right)u,u} \right\rangle ^{2^{n - 1}  - 1},
	\nonumber
	\end{align}
	is valid for all positive integer $n$.
	
	For $n=1$, and as we did in the proof of Lemma \ref{lem} taking into consideration the given assumptions, since $y\in \mathscr{H}$ is arbitrary then there exists
	$u\in \mathscr{H}$ such that   $y=Cu$, and this is true for any
	$y\in \mathscr{H}$. Therefore using \eqref{kittaneh.ineq} we have
	\begin{align*}
	\left| {\left\langle {ABx,Cu} \right\rangle } \right|^2 &\le
	\left\| {f\left( {\left| A \right|} \right)Bx} \right\|^2\left\|
	{g\left( {\left| {A^* } \right|} \right)Cu} \right\|^2
	\\
	&\le \left\langle {f\left( {\left| A \right|} \right)Bx,f\left(
		{\left| A \right|} \right)Bx} \right\rangle \left\langle {g\left(
		{\left| {A^* } \right|} \right)Cu,g\left( {\left| {A^* } \right|}
		\right)Cu} \right\rangle
	\\
	&\le\left\langle {B^* f^2 \left( {\left| A \right|} \right)Bx,x}
	\right\rangle \left\langle {C^* g^2 \left( {\left| {A^* } \right|}
		\right)Cu,u} \right\rangle
	\\
	&\le \left\langle {f^2 \left( {\left| A \right|} \right)B^2 x,x}
	\right\rangle \left\langle {g^2 \left( {\left| {A^* } \right|}
		\right)C^2 u,u} \right\rangle,
	\end{align*}
	where we used the fact that
	$ f^2 \left( {\left| A \right|} \right)B = B^* f^2 \left( {\left| A \right|} \right)$ and $   g^2 \left( {\left| A^* \right|} \right)C = C^*g^2 \left( {\left| A^* \right|} \right)$ for any nonnegative continuous  functions $f,g$ on $\left[0,\infty\right)$. This shows that the inequality valid for $n=1$.
	
	Assume the inequality holds for $n>1$, and let us check its validity for $n+1$. So that for all $x,u\in \mathscr{H}$ we have
	\begin{align*}
	&\left| {\left\langle {ABx,Cu} \right\rangle } \right|^{2^{n + 1}
	}
	\\
	&= \left( {\left| {\left\langle {ABx,Cu} \right\rangle }
		\right|^{2^n } } \right)^2
	\\
	&\le \left\{\left\langle {f^2 \left( {\left| A \right|}
		\right)B^{2^n } x,x} \right\rangle \left\langle {f^2 \left(
		{\left| A \right|} \right)x,x} \right\rangle ^{2^{n - 1}  -
		1}\right.
	\\
	&\qquad\left.\times \left\langle {g^2 \left( {\left| A^* \right|}
		\right)C^{2^n } u,u} \right\rangle \left\langle {g^2 \left(
		{\left| {A^* } \right|} \right)u,u} \right\rangle ^{2^{n - 1}  -
		1} \right\}^2
	\\
	&\le \left\langle {f^2 \left( {\left| A \right|} \right)B^{2^n }
		x,B^{2^n }x} \right\rangle \left\langle {f^2 \left( {\left| A
			\right|} \right)x,x} \right\rangle\left\langle {f^2 \left( {\left|
			A \right|} \right)x,x} \right\rangle ^{2^{n}  - 2}
	\\
	&\qquad\times \left\langle {g^2 \left( {\left| A^* \right|}
		\right)C^{2^n } u,C^{2^n }u} \right\rangle  \left\langle {g^2
		\left( {\left| A^* \right|} \right)u,u} \right\rangle \left\langle
	{g^2 \left( {\left| A^* \right|} \right)u,u} \right\rangle ^{2^{n
		}  - 2}  \qquad ({\rm{by}} \eqref{eq1.1})
	\\
	&= \left\langle {{B^*}^{2^n }f^2 \left( {\left| A \right|}
		\right)B^{2^n } x,x} \right\rangle  \left\langle {f^2 \left(
		{\left| A \right|} \right)x,x} \right\rangle ^{2^{n}  -1}
	\\
	&\qquad \times \left\langle {{C^*}^{2^n }g^2 \left( {\left| A^*
			\right|} \right)C^{2^n } u,u} \right\rangle   \left\langle {g^2
		\left( {\left| A^* \right|} \right)u,u} \right\rangle ^{2^{n }  -
		1}
	\\
	&= \left\langle {f^2 \left( {\left| A \right|} \right)B^{2^{n+1} }
		x,x} \right\rangle  \left\langle {f^2 \left( {\left| A \right|}
		\right)x,x} \right\rangle ^{2^{n}  -1}
	\\
	&\qquad\times \left\langle { g^2 \left( {\left| A^* \right|}
		\right)C^{2^{n+1} } u,u} \right\rangle   \left\langle {g^2 \left(
		{\left| A^* \right|} \right)u,u} \right\rangle ^{2^{n }  - 1},
	\end{align*}
	the last equation follows since
	$  f^2 \left( {\left| A \right|} \right)B^{2^n} = {B^*}^{2^n} f^2 \left( {\left| A \right|} \right)$ and $     g^2 \left( {\left| A^* \right|} \right)C^{2^n} = {C^*}^{2^n}g^2 \left( {\left| A^* \right|} \right)$. Hence, the inequality is true for $n+1$, and by Mathematical induction the inequality is completely proved for any $n\in \mathbb{N}$.
	
	Now, following Halmos approach in proving the generalized Reid inequality in \cite{H} (and Kittaneh as well in \cite{FK4}), we have from \eqref{eq2.3}
	\begin{align*}
	\left| {\left\langle {ABx,Cu} \right\rangle } \right|^{2^n }  &\le \left\| {f^2 \left( {\left| A \right|} \right)} \right\|\left\| {B^{2^n } } \right\|\left\| x \right\|^2 \left\langle {f^2 \left( {\left| A \right|} \right)x,x} \right\rangle ^{2^{n - 1}  - 1}  \\
	&\qquad\times \left\| {g^2 \left( {\left| {A^* } \right|} \right)}
	\right\|\left\| {C^{2^n } } \right\|\left\| u \right\|^2
	\left\langle {g^2 \left( {\left| {A^* } \right|} \right)u,u}
	\right\rangle ^{2^{n - 1}  - 1},
	\end{align*}
	and so that
	\begin{align*}
	\left| {\left\langle {ABx,Cu} \right\rangle } \right| &\le \left\| {f^2 \left( {\left| A \right|} \right)} \right\|^{\frac{1}{{2^n }}} \left\| {B^{2^n } } \right\|^{\frac{1}{{2^n }}} \left\| x \right\|^{\frac{2}{{2^n }}} \left\langle {f^2 \left( {\left| A \right|} \right)x,x} \right\rangle ^{\frac{1}{2} - \frac{1}{{2^n }}}  \\
	&\qquad\times \left\| {g^2 \left( {\left| {A^* } \right|} \right)}
	\right\|^{\frac{1}{{2^n }}} \left\| {C^{2^n } }
	\right\|^{\frac{1}{{2^n }}} \left\| u \right\|^{\frac{2}{{2^n }}}
	\left\langle {g^2 \left( {\left| {A^* } \right|} \right)u,u}
	\right\rangle ^{\frac{1}{2} - \frac{1}{{2^n }}}.
	\end{align*}
	Letting $n\longrightarrow \infty$, we obtain the desired result \eqref{eq2.2}.
\end{proof}

\begin{corollary}
	\label{cor1}    Let  $A, C\in \mathscr{B}\left( \mathscr{H}\right)$ such that
	$|A^*|C=C^*|A^*|$. If $f$ and $g$ as in Theorem \ref{thm1}, then
	\begin{align*}
	\left| {\left\langle {Ax,Cu} \right\rangle } \right| \le  r\left(C\right)\left\| {f\left( {\left| A \right|} \right)x} \right\|\left\| {g\left( {\left| {A^* } \right|} \right)u} \right\|
	\end{align*}
	for any   vectors $x,u\in  \mathscr{H} $.
\end{corollary}
\begin{proof}
	Setting $B=1_{\mathscr{H}}$ in \eqref{eq2.2} we get the required
	result.
\end{proof}

\begin{corollary}
	\label{cor2}    Let  $A,B,C\in \mathscr{B}\left( \mathscr{H}\right)$ such that
	$|A|B=B^*|A|$ and $|A^*|C=C^*|A^*|$.     Then
	\begin{align}
	\left| {\left\langle {ABx,Cu} \right\rangle } \right|^2 \le  r^2\left(B\right)r^2\left(C\right)
	\left\langle {\left| A \right|^{2\alpha } x,x} \right\rangle \left\langle {\left| {A^* } \right|^{2\left( {1 - \alpha } \right)} u,u} \right\rangle  \label{eq2.4}
	\end{align}
	for any   vectors $x,u\in  \mathscr{H} $. In particular we have
	\begin{align*}
	\left| {\left\langle {Bx,Cu} \right\rangle } \right|\le
	r\left(B\right)r\left(C\right).
	\end{align*}
\end{corollary}

\begin{proof}
	Setting $f(t)=t^{\alpha}$ and $g(t)=t^{1-\alpha}$, $0\le\alpha\le 1$, $t\ge0$ in Theorem \ref{thm1}. The particular case follows by setting $A=1_{\mathscr{H}}$ in \eqref{eq2.4} 
\end{proof}

A more general mixed Schwarz inequality can be stated as follows:

\begin{corollary}
	\label{cor3}     Let  $A,D,B_1,B_2,C_1,C_2\in \mathscr{B}\left( \mathscr{H}\right)$ such that
	\begin{align*}
	&|A|B_1=B_1^*|A|   \qquad {\rm{and}}  \qquad  |A^*|C_1=C_1^*|A^*|,\\
	&|D|B_2=B_2^*|D|  \qquad  {\rm{and}}  \qquad  |D^*|C_2=C_2^*|D^*|.
	\end{align*}
	If $f$ and $g$ are nonnegative continuous functions on $\left[0,\infty\right)$ satisfying $f(t)g(t) =t$ $(t\ge0)$, then
	\begin{align}
	&\left| {\left\langle {\left(C_1^*AB_1+C_2^*DB_2 \right)x, u} \right\rangle } \right|    \nonumber\\
	&\le r\left(B_1\right)r\left(C_1\right)\left\| {f\left( {\left| A \right|} \right)x} \right\|\left\| {g\left( {\left| {A^* } \right|} \right)u} \right\|\label{eq2.5}\\
	&\qquad+r\left(B_2\right)r\left(C_2\right)\left\| {f\left( {\left| D \right|} \right)x} \right\|\left\| {g\left( {\left| {D^* } \right|} \right)u} \right\| \nonumber
	\nonumber\\
	&\le \max\left\{r\left(B_1\right)r\left(C_1\right),r\left(B_2\right)r\left(C_2\right) \right\}\cdot \left(\left\| {f\left( {\left| A \right|} \right)x} \right\|^p+    \left\| {f\left( {\left| D \right|} \right)x} \right\|^p\right)^{1/p}
	\nonumber\\
	&\qquad\times  \left(\left\| {g\left( {\left| {A^* } \right|} \right)u} \right\|^q+\left\| {g\left( {\left| {D^* } \right|} \right)u} \right\|^q \right)^{1/q}   \nonumber
	\end{align}
	for any   vectors $x,u\in  \mathscr{H} $ and all $p>1$ with $q=\frac{p}{p-1}$.
\end{corollary}

\begin{proof}
	Since
	\begin{align*}
	\left| {\left\langle {\left(C_1^*AB_1+C_2^*DB_2 \right)x, u}
		\right\rangle } \right|&=\left| {\left\langle {C_1^*AB_1x, u}
		\right\rangle + \left\langle {C_2^*DB_2 x, u} \right\rangle }
	\right|
	\\
	&\le \left| {\left\langle {C_1^*AB_1x, u} \right\rangle } \right|
	+\left| {\left\langle {C_2^*DB_2 x, u} \right\rangle } \right|.
	\end{align*}
	So that  the first inequality follows from \eqref{eq2.2}. The
	second inequality follows by applying the H\"{o}lder inequality.
\end{proof}

\begin{corollary}
	\label{cor4} Let  $A,D,B_1,B_2,C_1,C_2\in \mathscr{B}\left( \mathscr{H}\right)$ such that
	\begin{align*}
	&|A|B_1=B_1^*|A|   \qquad {\rm{and}}  \qquad  |A^*|C_1=C_1^*|A^*|,\\
	&|D|B_2=B_2^*|D|  \qquad  {\rm{and}}  \qquad  |D^*|C_2=C_2^*|D^*|.
	\end{align*}
	If $f$ and $g$ are nonnegative continuous functions on $\left[0,\infty\right)$ satisfying $f(t)g(t) =t$ $(t\ge0)$, then
	\begin{align}
	\left| {\left\langle {\left(C_1^*AB_1+C_2^*DB_2 \right)x, u}
		\right\rangle } \right|^2 &\le
	r^2\left(B_1\right)r^2\left(C_1\right)
	\left\langle {\left| A \right|^{2\alpha } x,x} \right\rangle \left\langle {\left| {A^* } \right|^{2\left( {1 - \alpha } \right)} u,u} \right\rangle     \label{eq2.6}
	\\
	&\qquad+ r^2\left(B_2\right)r^2\left(C_2\right)
	\left\langle {\left| D \right|^{2\alpha } x,x} \right\rangle \left\langle {\left| {D^* } \right|^{2\left( {1 - \alpha } \right)} u,u} \right\rangle \nonumber
	\end{align}
	for any   vectors $x,u\in  \mathscr{H} $. 
\end{corollary}

\begin{proof}
	Setting $f(t)=t^{\alpha}$ and $g(t)=t^{1-\alpha}$, $0\le\alpha\le 1$, $t\ge0$ in Corollary \ref{cor3}.  
\end{proof}

In fact, one may establish a generalization of Corollary
\ref{cor3}  to several operators, by letting $A_i,B_i,C_i\in
\mathscr{B}\left( \mathscr{H}\right)$ $(i=1,\cdots,n)$ such that
\begin{align*}
|A_i|B_i=B_i^*|A_i|   \qquad {\rm{and}}  \qquad
|A_i^*|C_i=C_i^*|A_i^*|.
\end{align*}
If $f,g$ are as above, proceeding as in the presented proof above,
then we have
\begin{align}
&\left| {\left\langle {\left(\sum_{i=1}^{n}C_i^*A_iB_i\right)x, u} \right\rangle } \right| \nonumber \\
&\le \sum_{i=1}^n{r\left(B_i\right)r\left(C_i\right) \left\| {f\left( {\left| A_i \right|} \right)x} \right\|\left\| {g\left( {\left| {A_i^* } \right|} \right)u} \right\|} \label{eq2.7}\\
&\le \max_{1\le i \le n}\left\{r\left(B_i\right)r\left(C_i\right)
\right\}\cdot \left(\sum_{i=1}^n \left\| {f\left( {\left| A_i
		\right|} \right)x} \right\|^p \right)^{1/p}  \left(\sum_{i=1}^n
\left\| {g\left( {\left| {A_i^* } \right|} \right)u} \right\|^q
\right)^{1/q}.\nonumber
\end{align}
for all $x,u\in \mathscr{H}$, which follows by the properties of `$\max$' and H\"{o}lder inequality, where $p,q$ are conjugate exponents, i.e.,  $p,q>1$ with $\frac{1}{p}+\frac{1}{q}=1$.\\

Thus, one may has the following norm inequality
\begin{multline}
\left\| \sum_{i=1}^{n}C_i^*A_iB_i \right\| \\ \le  \max_{1\le i
	\le n}\left\{r\left(B_i\right)r\left(C_i\right) \right\}\cdot
\left(\sum_{i=1}^n \left\| {f\left( {\left| A_i \right|} \right) }
\right\|^p \right)^{1/p}  \left(\sum_{i=1}^n \left\| {g\left(
	{\left| {A_i^* } \right|} \right)} \right\|^q
\right)^{1/q}.\label{eq2.8}
\end{multline}
For instance,  consider $f(t)=t^{\alpha}$ and $g(t)=t^{1-\alpha}$,
one has from  \eqref{eq2.8} that
\begin{align*}
\left\| \sum_{i=1}^{n}C_i^*A_iB_i \right\|  \le  \max_{1\le i \le
	n}\left\{r\left(B_i\right)r\left(C_i\right) \right\}\cdot
\left(\sum_{i=1}^n \left\| {\left| A_i \right|^{\alpha}  }
\right\|^p \right)^{1/p}  \left(\sum_{i=1}^n \left\| { \left|
	{A_i^* } \right|^{1-\alpha} } \right\|^q \right)^{1/q}.
\end{align*}
Also, if $C_i=B_i=1_{\mathscr{H}}$ for all $i=1,\cdots,n$, then
the last inequality reduces to
\begin{align*}
\left\| \sum_{i=1}^{n} A_i  \right\|  \le  \left(\sum_{i=1}^n
\left\| {\left| A_i \right|^{\alpha}  } \right\|^p \right)^{1/p}
\left(\sum_{i=1}^n \left\| { \left| {A_i^* } \right|^{1-\alpha} }
\right\|^q \right)^{1/q}.
\end{align*}

\subsection{The  Mixed hybrid Schwarz inequality\label{sec2.2}}

Merging   both Cartesian and  Polar decompositions of operators
will produce  a new hybrid    Mixed Schwarz inequality including
both decompositions. The   next result provides a new extension of
the mixed Schwarz inequality \eqref{eq1.3} and their
generalizations  \eqref{kittaneh.ineq} and \eqref{eq2.2}.
\begin{theorem}
	\label{thm2}Let  $A\in \mathscr{B}\left( \mathscr{H}\right)$
	with the Cartesian decomposition    $A=P+iQ$.  If $f$ and $g$ are as in
	Theorem \ref{thm1}. Then
	\begin{align}
	\left| {\left\langle {A x,y} \right\rangle } \right| \le  \left\{ {\left\| {f\left( \left|P\right| \right)x}
		\right\|\left\| {g\left( \left|P \right| \right)y} \right\| +
		\left\| {f\left( \left|Q\right| \right)x} \right\|\left\| {g\left(
			\left|Q\right| \right)y} \right\|} \right\} \label{eq2.9}
	\end{align}
	for all $x,y\in \mathscr{H}$.
\end{theorem}
\begin{proof}
	Let $ P+iQ$ be  the Cartesian decomposition of $A$. Setting $B=C=1_{\mathscr{H}}$ in \eqref{eq2.2}, then
	\begin{align*}
	\left| {\left\langle {Ax,y} \right\rangle } \right| &= \left( {\left\langle { Px,y} \right\rangle ^2  + \left\langle {Qx,y} \right\rangle ^2 } \right)^{1/2}  \\
	&\le \left| {\left\langle { P x,y} \right\rangle } \right| + \left| {\left\langle {Q x,y} \right\rangle } \right| \\
	&\le  \left\{ {\left\| {f\left( \left|P\right|
			\right)x} \right\|\left\| {g\left( \left|P^*\right| \right)y}
		\right\| + \left\| {f\left( \left|Q\right| \right)x}
		\right\|\left\| {g\left( \left|Q^*\right| \right)y} \right\|}
	\right\},
	\end{align*}
	for all $x,y\in \mathscr{H}$, where the last inequality follows form \eqref{eq2.2}. So that, the
	required result follows since  $P$ and $Q$ are selfadjoint
	operators.
\end{proof}

\begin{corollary}
	\label{cor5}    Let  $A\in \mathscr{B}\left( \mathscr{H}\right)$ with the Cartesian decomposition    $A=P+iQ$.  Then,
	\begin{align}
	\left| {\left\langle {Ax,y} \right\rangle } \right|
	\le  \left\{ {\left\| { \left|P\right|^{2\alpha}x} \right\|\left\| { \left|P \right|^{2\left(1-\alpha\right)} y} \right\| + \left\| { \left|Q\right|^{2\alpha}x} \right\|\left\| { \left|Q\right|^{2\left(1-\alpha\right)} y} \right\|} \right\} \label{eq2.10}
	\end{align}
	for all $x,y\in \mathscr{H}$.
\end{corollary}
\begin{proof}
	Setting $f(t)=t^{\alpha}$ and $g(t)=t^{1-\alpha}$, $0\le \alpha \le 1$, $t\ge0$ in Theorem \ref{thm2}  we get \eqref{eq2.10}.
\end{proof}

The Cartesian companion decomposition of Kato's inequality
\eqref{eq1.3} can be deduced as follows:
\begin{corollary}
	\label{cor6}Let  $A  \in \mathscr{B}\left( \mathscr{H}\right)$
	with the Cartesian decomposition    $A=P+iQ$.  Then,
	\begin{align}
	\left| {\left\langle {A x, y} \right\rangle } \right|^2 \le
	\left\langle {\left| P \right|^{2\alpha } x,x} \right\rangle
	\left\langle {\left| P \right|^{2\left( {1 - \alpha } \right)}
		y,y} \right\rangle +\left\langle {\left| Q \right|^{2\alpha } x,x}
	\right\rangle \left\langle {\left| Q \right|^{2\left( {1 - \alpha
			} \right)} y,y} \right\rangle  \label{eq2.11}
	\end{align}
	for all $x,y\in \mathscr{H}$ and any $0\le \alpha\le 1$.
\end{corollary}
\begin{proof}
	Setting $f(t)=t^{\alpha}$ and $g(t)=t^{1-\alpha}$, $0\le \alpha
	\le 1$, $t\ge0$ in Corollary \ref{cor5}, then we have
	\begin{align*}
	\left| {\left\langle {A x, y} \right\rangle } \right|^2 &\le
	\left\| { \left|P\right|^{ \alpha}x} \right\|^2\left\| { \left|P
		\right|^{  1-\alpha } y} \right\|^2 + \left\| { \left|Q\right|^{
			\alpha}x} \right\|^2 \left\| { \left|Q\right|^{ 1-\alpha } y}
	\right\|^2
	\\
	&=\left\langle {\left| P \right|^{2\alpha } x,x} \right\rangle
	\left\langle {\left| P \right|^{2\left( {1 - \alpha } \right)}
		y,y} \right\rangle +\left\langle {\left| Q \right|^{2\alpha } x,x}
	\right\rangle \left\langle {\left| Q \right|^{2\left( {1 - \alpha
			} \right)} y,y} \right\rangle,
	\end{align*}
	for all $x,y\in \mathscr{H}$, which proves the required result.
\end{proof}

\begin{remark}
	Some Weyl type inequalities can be deduced by following the same approach cosidered in \cite{FK4}. In fact  by making use of    \eqref{eq2.2} and \eqref{eq2.9} instead of \eqref{kittaneh.ineq} in  \cite{FK4}, a general Weyl type inequality can be deduced. Similarly, some inequalities for the $p$-Schatten norm can be pointed out following the same pattern in \cite{FK4}. We shall omit the details. 
\end{remark}

\section{Numerical Radius inequalities}\label{sec3}

In order to prove our  results in this section we need some of the following well-known facts.
\begin{lemma}
	\label{lemma1}
	The Power-Young inequality reads that
	\begin{align}
	\label{YI}ab \le \frac{{a^{\alpha} }}{\alpha} + \frac{{b^{\beta} }}{\beta} \le \left( {\frac{{a^{p\alpha} }}{\alpha} + \frac{{b^{p\beta} }}{\beta}} \right)^{\frac{1}{p}}
	\end{align}
	for all $a,b\ge0$ and $\alpha,\beta>1$ with $\frac{1}{\alpha}+\frac{1}{\beta}=1$ and all $p\ge1$.
\end{lemma}
\begin{lemma} {\rm{(The McCarty inequality).}}
	\label{lemma2}  Let  $A\in \mathscr{B}\left( \mathscr{H}\right)^+ $, then
	\begin{align}
	\left\langle {Ax,x} \right\rangle ^p  \le \left\langle {A^p x,x} \right\rangle, \qquad p\ge1, \label{mc1}
	\end{align}
	for any unit vector $x\in\mathscr{H}$
\end{lemma}

\begin{lemma}
	\label{lemma6}If $A,B\in \mathscr{B}\left( \mathscr{H}\right)$. Then
	\begin{align}
	r\left( {AB} \right) \label{eq3.3}\le \frac{1}{4}\left( {\left\| {AB} \right\| + \left\| {BA} \right\| + \sqrt {\left(\left\| {AB} \right\|- \left\| {BA} \right\|\right)^2 + 4m\left(A,B\right)}} \right),
	\end{align}
	where $m\left(A,B\right):=\min \left\{ {\left\| A \right\|\left\| {BAB} \right\|,\left\| B \right\|\left\| {ABA} \right\|} \right\}$.
\end{lemma}

In some of our results we need the following two fundamental norm estimates, which  are:
\begin{align}
\label{eq3.4}\left\| {A+ B } \right\|\le
\frac{1}{2}\left( {\left\| A \right\| + \left\| B \right\| + \sqrt
	{\left( {\left\| A \right\| - \left\| B \right\|} \right)^2  +
		4\left\| {A^{1/2} B^{1/2} } \right\|^2 } } \right),
\end{align}
and
\begin{align}
\label{eq3.5}\left\| {A^{1/2} B^{1/2} } \right\|  \le\left\| {A  B } \right\| ^{1/2}.
\end{align}
Both estimates are valid for all positive operators $A,B \in \mathscr{B}\left( \mathscr{H}\right)$. Also, it should be noted that \eqref{eq3.4} is sharper than the triangle inequality as pointed out by Kittaneh in \cite{FK3}. \\

\subsection{Inequalities using The  Mixed Schwarz inequality}
Depending on the obtained results in Section \ref{sec2.1},  in
this part we provide some numerical radius inequalities. Let us
start with the following main result.
\begin{theorem}
	\label{thm3}   Let  $A,B,C\in \mathscr{B}\left( \mathscr{H}\right)$ such that
	$|A|B=B^*|A|$  and  $|A^*|C=C^*|A^*|$. If $f$ and $g$ are nonnegative continuous functions on $\left[0,\infty\right)$ satisfying $f(t)g(t) =t$ $(t\ge0)$. Then
	\begin{align}
	w \left( C^*AB \right) &\le \frac{1}{2} r \left(B \right)r \left(C \right)\cdot  \left\| { f^2\left( {\left| {A  } \right|} \right) +g^2\left( {\left| {A^*} \right|} \right) } \right\|\label{eq3.1}
	\\
	&\le\frac{1}{16} \left( {\left\| B \right\| + \left\| {B^2 } \right\|^{1/2} } \right)\left( {\left\| C \right\| + \left\| {C^2 } \right\|^{1/2} } \right)
	\nonumber\\
	&\qquad \times \left\{ \left\| {f^2 \left( {\left| A \right|} \right)} \right\| + \left\| {g^2 \left( {\left| {A^* } \right|} \right)} \right\|  \right.\nonumber\\
	&\qquad\left.+  \sqrt {\left( {\left\| {f^2 \left( {\left| A \right|} \right)} \right\| - \left\| {g^2 \left( {\left| {A^* } \right|} \right)} \right\|} \right)^2  + 4\left\| {f\left( {\left| A \right|} \right)g\left( {\left| {A^* } \right|} \right)} \right\|^2 } \right\}.\nonumber
	\end{align}
	In particular, we have
	\begin{align}
	w \left( C^*C \right) &\le \frac{1}{2} r \left(C \right)\cdot  \left\| { f^2\left( {\left| {C  } \right|} \right) +g^2\left( {\left| {C^*} \right|} \right) } \right\|\label{eq3.2}
	\\
	&\le\frac{1}{8}  \left( {\left\| C \right\| + \left\| {C^2 } \right\|^{1/2} } \right)\left\{ \left\| {f^2 \left( {\left| C \right|} \right)} \right\| + \left\| {g^2 \left( {\left| {C^* } \right|} \right)} \right\|  \right.\nonumber\\
	&\qquad\left.+  \sqrt {\left( {\left\| {f^2 \left( {\left| C \right|} \right)} \right\| - \left\| {g^2 \left( {\left| {C^* } \right|} \right)} \right\|} \right)^2  + 4\left\| {f\left( {\left| C \right|} \right)g\left( {\left| {C^* } \right|} \right)} \right\|^2 } \right\}.\nonumber
	\end{align}
\end{theorem}
\begin{proof}
	Setting $y=x$ in  \eqref{eq2.2}, we get
	\begin{align*}
	\left| {\left\langle {C^*ABx,x} \right\rangle } \right| &\le r\left(B\right)r\left(C\right)\left\| {f\left( {\left| A \right|} \right)x} \right\|\left\| {g\left( {\left| {A^* } \right|} \right)x} \right\| \\
	&=  r \left(B \right)r\left(C\right)\left\langle {f^2\left( {\left| {A  } \right|} \right)x  ,x  } \right\rangle^{1/2}  \left\langle {g^2\left( {\left| {A^*} \right|} \right)x  ,x  } \right\rangle^{1/2}    \\
	&\le \frac{1}{2} r \left(B \right)r\left(C\right) \left(\left\langle {f^2\left( {\left| {A  } \right|} \right)x  ,x  } \right\rangle +  \left\langle {g^2\left( {\left| {A^*} \right|} \right)x  ,x  } \right\rangle  \right)\\
	&=   \frac{1}{2} r \left(B \right)r\left(C\right) \left\langle {\left(f^2\left( {\left| {A  } \right|} \right) +g^2\left( {\left| {A^*} \right|} \right)\right)x  ,x  } \right\rangle  \\
	&= \frac{1}{2} r \left(B \right)r\left(C\right) \left\| {\left(f^2\left( {\left| {A  } \right|} \right) +g^2\left( {\left| {A^*} \right|} \right)\right)} \right\|.
	\end{align*}
	Thus,   by taking the supremum over $x\in \mathscr{H}$  we get the first inequality in \eqref{eq3.1}. The second inequality in \eqref{eq3.1} follows by employing   \eqref{eq3.3} on the first inequality and  use \eqref{eq3.4}. The inequality \eqref{eq3.2} follows from \eqref{eq3.1} by setting $B=I$ and $A=C$.
\end{proof}

\begin{corollary}
	\label{cor8}    Let  $A,B,C\in \mathscr{B}\left( \mathscr{H}\right)$ such that
	$|A|B=B^*|A|$  and  $|A^*|C=C^*|A^*|$. Then
	\begin{align}
	w \left( C^*AB \right) &\le  \frac{1}{2} r \left(B \right)r \left(C \right)\cdot   \left\| {  \left| {A  } \right|^{2\alpha}  +\left| {A^*} \right|^{2\left(1-\alpha\right)}} \right\|\label{eq3.8}
	\\
	&\le\frac{1}{16} \left( {\left\| B \right\| + \left\| {B^2 } \right\|^{1/2} } \right)\left( {\left\| C \right\| + \left\| {C^2 } \right\|^{1/2} } \right)
	\nonumber\\
	&\qquad\times \left\{ \left\| { \left| A \right|^{2\alpha} } \right\| + \left\| { \left| {A^* } \right|^{2\left(1-\alpha\right)} } \right\|  \right.\nonumber\\
	&\qquad\left.+  \sqrt {\left( {\left\| { \left| A \right|^{ 2\alpha} } \right\| - \left\| { \left| {A^* } \right|^{2\left(1-\alpha\right)} } \right\|} \right)^2  + 4\left\| { \left| A \right|^{\alpha} \left| {A^* } \right|^{1-\alpha} } \right\|^2 } \right\}.\nonumber
	\end{align}
	
\end{corollary}
\begin{proof}
	Setting $f(t)=t^{\alpha}$ and $g(t)=t^{1-\alpha}$, $0\le \alpha \le 1$, $t\ge0$ in Theorem \ref{thm3}  we get \eqref{eq3.8}.
\end{proof}

\begin{remark}
	Setting $\alpha=\frac{1}{2}$ in \eqref{eq3.8} and then employing
	\eqref{eq3.5} we get
	\begin{align}
	w \left( C^*AB \right) \le\frac{1}{8} \left( {\left\| B \right\| + \left\| {B^2 } \right\|^{1/2} } \right)\left( {\left\| C \right\| + \left\| {C^2 } \right\|^{1/2} } \right) \left( {\left\| A \right\| + \left\| {A^2 } \right\|^{1/2} } \right)
	\end{align}
	
\end{remark}

\begin{remark}
	Letting $A=C$ and $B=1_{\mathscr{H}}$ in \eqref{eq3.8}. Then
	\begin{align*}
	w \left( C^*C \right) &\le  \frac{1}{2}  r \left(C \right)\cdot
	\left\| {  \left| {C  } \right|^{2\alpha}  +\left| {C^*}
		\right|^{2\left(1-\alpha\right)}} \right\|
	\\
	&\le\frac{1}{8}  \left( {\left\| C \right\| + \left\| {C^2 } \right\|^{1/2} } \right)\left\{ \left\| { \left| C \right|^{2\alpha} } \right\| + \left\| { \left| {C^* } \right|^{2\left(1-\alpha\right)} } \right\|  \right.\nonumber\\
	&\qquad\left.+  \sqrt {\left( {\left\| { \left| C \right|^{
					2\alpha} } \right\| - \left\| { \left| {C^* }
				\right|^{2\left(1-\alpha\right)} } \right\|} \right)^2  + 4\left\|
		{ \left| C \right|^{\alpha} \left| {C^* } \right|^{1-\alpha} }
		\right\|^2 } \right\}.\nonumber
	\end{align*}
	Moreover,  setting $\alpha=\frac{1}{2}$ in the above inequality  and   use \eqref{eq3.5} with the fact that $\left\| { \left| C \right|  } \right\| = \left\| { \left| {C^* }\right| } \right\| = \|C\|$. So that we get
	\begin{align*}
	w\left( C^* C\right) \le   \frac{1}{4}  \left( \left\| C \right\|
	+ \left\| {C^2 } \right\|^{1/2} \right)^2.
	\end{align*}
\end{remark}

\begin{corollary}
	\label{cor9}  Let  $A,B,C\in \mathscr{B}\left( \mathscr{H}\right)$
	such that $|A|C=C^*|A|$  and  $|A^*|C=C^*|A^*|$. If $f$ and $g$
	are nonnegative continuous functions on $\left[0,\infty\right)$
	satisfying $f(t)g(t) =t$ $(t\ge0)$. Then
	\begin{align*}
	w \left( C^*AC \right) &\le \frac{1}{2} r^2 \left(C \right)\cdot  \left\| { f^2\left( {\left| {A  } \right|} \right) +g^2\left( {\left| {A^*} \right|} \right) } \right\|
	\\
	&\le\frac{1}{16} \left( {\left\| C \right\| + \left\| {C^2 } \right\|^{1/2} } \right)^2\left\{ \left\| {f^2 \left( {\left| A \right|} \right)} \right\| + \left\| {g^2 \left( {\left| {A^* } \right|} \right)} \right\|  \right.\nonumber\\
	&\qquad\left.+  \sqrt {\left( {\left\| {f^2 \left( {\left| A \right|} \right)} \right\| - \left\| {g^2 \left( {\left| {A^* } \right|} \right)} \right\|} \right)^2  + 4\left\| {f\left( {\left| A \right|} \right)g\left( {\left| {A^* } \right|} \right)} \right\|^2 } \right\}.\nonumber
	\end{align*}
\end{corollary}
\begin{proof}
	Setting $B=C$ in Theorem \ref{thm3}.
\end{proof}

A generalization of Theorem \ref{thm3} to higher order power is given as
follows:
\begin{theorem}
	\label{thm4}   Let  $A,B\in \mathscr{B}\left(
	\mathscr{H}\right)\left(\Omega\right)$ such that
	$|A|B= B^*|A|$. If $f, g$ be nonnegative
	continuous functions   on $\left[0,\infty\right)$  satisfying
	$f(t)g(t) = t$, $(t \ge 0)$. Then
	\begin{align}
	w^{p} \left( C^*AB \right)  \le r^p\left(B\right)r^p\left(C\right)\cdot
	\left\| {\frac{1}{\alpha }f^{\alpha p} \left( {\left| {A  } \right|} \right) + \frac{1}{\beta }g^{\beta p} \left( {\left| {A^* } \right|} \right)} \right\|   \label{eq3.10}
	\end{align}
	for all $p\ge1$, $\alpha \ge \beta >1 $ with $\frac{1}{\alpha}+\frac{1}{\beta}=1$ and $\beta p \ge2$. Moreover we have
	\begin{align}
	\label{eq3.11}w^{p} \left( C^*AB \right)
	&\le \frac{1}{2^{p+2}} \cdot \gamma \cdot \left( \left\| B \right\| + \left\| {B^2 } \right\|^{1/2} \right)^p \left( \left\| C \right\| + \left\| {C^2 } \right\|^{1/2} \right)^p
	\nonumber\\
	&\qquad \times\left\{ {\left\| {f^{\alpha p} \left( {\left| {A  } \right|} \right)} \right\| + \left\| {g^{\beta p} \left( {\left| {A^* } \right|} \right)} \right\|} \right. \\
	&\qquad\left. { + \sqrt {\left[ {\left| {f^{\alpha p} \left( {\left| {A  } \right|} \right)} \right| - \left\| {g^{\beta p} \left( {\left| {A^* } \right|} \right)} \right\|} \right]^2  + 4\left\| {f^{p\alpha } \left( {\left| {A  } \right|} \right)g^{p\beta } \left( {\left| {A^* } \right|} \right)} \right\|^2 } } \right\},\nonumber
	\end{align}
	where $\gamma= \max\{\frac{1}{\alpha},\frac{1}{\beta}\}$.
	
\end{theorem}

\begin{proof}
	Setting $u=x$ in the generalized mixed Schwarz inequality \eqref{eq2.2},  we have
	\begin{align*}
	&\left| {\left\langle {C^*ABx,x} \right\rangle } \right|^p \\&\le r^p\left(B\right)r^p\left(C\right)\left\| {f\left( {\left| A \right|} \right)x} \right\|^p\left\| {g\left( {\left| {A^* } \right|} \right)x} \right\|^p \\
	&=  r^p\left(B\right)r^p\left(C\right)\left\langle {f^2\left( {\left| {A } \right|} \right)x  ,x  } \right\rangle ^{\frac{p}{2}} \left\langle {g^2\left( {\left| {A^*} \right|} \right)x  ,x  } \right\rangle ^{\frac{p}{2}}  \\
	&\le r^p\left(B\right) r^p\left(C\right)\left[\frac{1}{\alpha }\left\langle {f^2 \left( {\left| {A } \right|} \right)x  ,x  } \right\rangle ^{\frac{{\alpha p}}{2}}  + \frac{1}{\beta }\left\langle {g^2 \left( {\left| {A^* } \right|} \right)x  ,x  } \right\rangle ^{\frac{{\beta p}}{2}}  \right] \qquad \text{(by \eqref{YI})}\\
	&\le  r^p\left(B\right) r^p\left(C\right) \left[\frac{1}{\alpha }\left\langle {f^{\alpha p} \left( {\left| {A  } \right|} \right)x  ,x  } \right\rangle  + \frac{1}{\beta }\left\langle {g^{\beta p} \left( {\left| {A^* } \right|} \right)x  ,x  } \right\rangle \right] \qquad \text{(by \eqref{mc1})}\\
	&=  r^p\left(B\right)  r^p\left(C\right)\left\langle {\left[ {\frac{1}{\alpha}f^{\alpha p} \left( {\left| {A   } \right|} \right) +\frac{1}{\beta }g^{\beta p} \left( {\left| {A^* }\right|} \right)} \right]x  ,x } \right\rangle.
	\end{align*}
	Taking the supremum over $x \in \mathscr{H}$, we obtain
	the   inequality in \eqref{eq3.10}. To obtain the second inequality, by \eqref{eq3.4} we have
	\begin{align*}
	&\left\| {\frac{1}{\alpha }f^{\alpha p} \left( {\left| {A  } \right|} \right) + \frac{1}{\beta }g^{\beta p} \left( {\left| {A^* } \right|} \right)} \right\|
	\\
	&\le \max\{\frac{1}{\alpha},\frac{1}{\beta}\}\cdot \left\| { f^{\alpha p} \left( {\left| {A  } \right|} \right) +  g^{\beta p} \left( {\left| {A^* } \right|} \right)} \right\|
	\\
	&\le \frac{1}{2}\gamma\left( {\left\| {f^{\alpha p} \left( {\left| {A  } \right|} \right)} \right\| + \left\| {g^{\beta p} \left( {\left| {A^* } \right|} \right)} \right\|} \right. \\
	&\qquad\left. { + \sqrt {\left[ {\left| {f^{\alpha p} \left( {\left| {A  } \right|} \right)} \right| - \left\| {g^{\beta p} \left( {\left| {A^* } \right|} \right)} \right\|} \right]^2  + 4\left\| {f^{p\alpha } \left( {\left| {A  } \right|} \right)g^{p\beta } \left( {\left| {A^* } \right|} \right)} \right\|^2 } } \right).
	\end{align*}
	Now, employing \eqref{eq3.3} with $B=1_{\mathscr{H}}$ and then substituting all in  \eqref{eq3.10} we get \eqref{eq3.11}.
	
\end{proof}

\begin{remark}
	Letting $u=x$   in \eqref{eq2.7},  then by  taking the supremum over $x\in \mathscr{H}$ with $\|x\| =1$ so that we get
	\begin{align}
	&w\left(\sum_{i=1}^{n}C_i^*A_iB_i\right)    \\
	&\le \sum_{i=1}^n{r\left(B_i\right)r\left(C_i\right) \left\| {f\left( {\left| A_i \right|} \right) } \right\|\left\| {g\left( {\left| {A_i^* } \right|} \right) } \right\|} \nonumber\\
	&\le \max_{1\le i \le n}\left\{r\left(B_i\right)r\left(C_i\right) \right\}\cdot \left(\sum_{i=1}^n \left\| {f\left( {\left| A_i \right|} \right) } \right\|^p \right)^{1/p}  \left(\sum_{i=1}^n \left\| {g\left( {\left| {A_i^* } \right|} \right)} \right\|^q \right)^{1/q}.  \nonumber
	\end{align}
	Following the same approach  applied in the proof of Theorems \ref{thm3} and \ref{thm4}, one can state other bounds  for the second  inequality above. Several special cases can also be obtained as in the Corollaries \ref{cor8} and \ref{cor9}. Of course the same inequalities still valid for norms instead of numerical radius.
\end{remark}

\subsection{Inequalities using the  Mixed hybrid Schwarz inequality}

\begin{theorem}
	\label{thm5}Let  $A \in \mathscr{B}\left( \mathscr{H}\right)$
	with the Cartesian decomposition    $A=P+iQ$.   If $f$ and $g$ are as in
	Theorem \ref{thm1}. Then
	\begin{align}
	w \left(A\right)  &\le   \left\| {f^p \left(
		{\left| P \right|} \right) + f^p \left( {\left| Q \right|}
		\right)} \right\|^{1/p} \left\| {g^q \left( {\left| P \right|}
		\right) + g^q \left( {\left| Q \right|} \right)} \right\|^{1/q}
	\label{eq3.13}
	\end{align}
	for all $p,q\ge2$ with $\frac{1}{p}+\frac{1}{q}=1$.
\end{theorem}

\begin{proof}
	Letting $y=x$ in  \eqref{eq2.8}, then we have
	\begin{align*}
	&\left| {\left\langle {Ax,y} \right\rangle } \right|   \\
	&\le     \left\{ {\left\| {f\left( {\left| P \right|} \right)x} \right\|\left\| {g\left( {\left| P \right|} \right)y} \right\| + \left\| {f\left( {\left| Q \right|} \right)x} \right\|\left\| {g\left( {\left| Q \right|} \right)y} \right\|} \right\}
	\\
	&\le   
	\left( {\left\| {f\left( {\left| P \right|} \right)x} \right\|^p  + \left\| {f\left( {\left| Q \right|} \right)x} \right\|^p } \right)^{1/p}
	\\
	&\qquad\qquad\times\left( {\left\| {g\left( {\left| P \right|} \right)y} \right\|^q  + \left\| {g\left( {\left| Q \right|} \right)y} \right\|^q } \right)^{1/q} \nonumber  \qquad {(\rm{by\,\,H\text{\"{o}}lder\,\, inequaity})}\\
	&\le \left( {\left\langle {f^2 \left( {\left| P \right|} \right)x,x} \right\rangle ^{p/2}  + \left\langle {f^2 \left( {\left| Q \right|} \right)x,x} \right\rangle ^{p/2} } \right)^{1/p}
	\\
	&\qquad \times\left( {\left\langle {g^2 \left( {\left| P \right|} \right)x,x} \right\rangle ^{q/2}  + \left\langle {g^2 \left( {\left| Q \right|} \right)x,x} \right\rangle ^{q/2} } \right)^{1/q}  \nonumber\\
	&\le \left( {\left\langle {f^p \left( {\left| P \right|} \right)x,x} \right\rangle  + \left\langle {f^p \left( {\left| Q \right|} \right)x,x} \right\rangle } \right)^{1/p}
	\\
	&\qquad\times \left( {\left\langle {g^q \left( {\left| P \right|} \right)x,x} \right\rangle  + \left\langle {g^q \left( {\left| Q \right|} \right)x,x} \right\rangle } \right)^{1/q}  \qquad\qquad{\text({\rm{by}}\,\, \eqref{mc1})}\nonumber\\
	&\le \left\langle {\left[ {f^p \left( {\left| P \right|} \right) + f^p \left( {\left| Q \right|} \right)} \right]x,x} \right\rangle ^{1/p} \left\langle {\left[ {g^q \left( {\left| P \right|} \right) + g^q \left( {\left| Q \right|} \right)} \right]x,x} \right\rangle ^{1/q}  \nonumber
	\end{align*}
	for all $p,q\ge2$ with $\frac{1}{p}+\frac{1}{q}=1$. Taking the
	supremum over all unit vector   $x\in \mathscr{H}$ we get the
	desired result.
	
\end{proof}

However, we can still have a little more manipulation; by
employing \eqref{eq3.4} for the above two norms we get
\begin{align}
&\left\| {\left(f^{p} \left( {\left| P \right|} \right) + f^{p}
	\left( {\left| Q \right|} \right)\right)}\right\|
\nonumber\\
&\le \frac{1}{2} \left(\left\| {f^{p} \left( {\left| P \right|}
	\right)} \right\| + \left\| {f^{p} \left( {\left| Q \right|}
	\right)} \right\|\right. \label{eq3.14}
\\
&\qquad\left.{+ \sqrt {\left( {\left\| {f^{p} \left( {\left| P
					\right|} \right)} \right\| - \left\| {f^{p} \left( {\left| Q
					\right|} \right)} \right\|} \right)^2  + 4\left\| {f^{p/2}
			\left({\left| P \right|} \right)f^{p/2} \left( {\left| Q \right|}
			\right)} \right\|^2} }\right),\nonumber
\end{align}
and
\begin{align}
&\left\| {\left(g^{q} \left( {\left| P \right|} \right) + g^{q}
	\left( {\left| Q \right|} \right)\right)}\right\|
\nonumber\\
&\le \frac{1}{2} \left(\left\| {g^{q} \left( {\left| P \right|}
	\right)} \right\| + \left\| {g^{q} \left( {\left| Q \right|}
	\right)} \right\|\right. \label{eq3.15}
\\
&\qquad\left.{+ \sqrt {\left( {\left\| {g^{q} \left( {\left| P
					\right|} \right)} \right\| - \left\| {g^{q} \left( {\left| Q
					\right|} \right)} \right\|} \right)^2  + 4\left\|
		{g^{q/2}
			\left({\left| P \right|} \right)g^{q/2} \left( {\left| Q \right|}
			\right)} \right\|^2} }\right).\nonumber
\end{align}

Substituting \eqref{eq3.14} and \eqref{eq3.15} in \eqref{eq3.13}
we get  another refinement of \eqref{eq3.13}.

\begin{remark}
	In  an interesting case, one  may  consider   $f(t)=t^\alpha$ and $g(t)=t^{1-\alpha}$, $\alpha\in [0,1]$ and $p=q=2$. If we wish  for $\alpha=\frac{1}{2}$, after some manipulations and making of use  \eqref{eq3.5} we get
	\begin{align*}
	w\left( {A} \right) \le \frac{1}{2} \cdot \left( {\left\| {\left| P \right|} \right\| + \left\| {\left| Q \right|} \right\| + \sqrt {\left( {\left\| {\left| P \right|} \right\| - \left\| {\left| Q \right|} \right\|} \right)^2  + 4\left\| {\left| P \right|\left| Q \right|} \right\|} } \right).
	\end{align*}
	It should be noted that the authors in    \cite{EF}  have shown that $ w\left( {A} \right) \le \left\| {\left| P \right|} \right\| + \left\| {\left| Q \right|} \right\| $, it is not hard to show that our  estimate is better than the previous one. 
\end{remark}
\begin{remark}
	Following the same approach considered in the proof of Theorem \ref{thm4}, one may state another bound of \eqref{eq3.13}.
\end{remark}


\bibliographystyle{amsplain}

\end{document}